\documentclass[10pt]{article}

\usepackage{amsmath,amssymb,amsthm}
\usepackage{bbm}
\usepackage{mathtools}

\newcommand{\as}{\\[.6em]}
\newcommand{\As}{\\[.9em]}
\newcommand{\AS}{\\[1.2em]}
\newcommand{\val}{\,\,\,}

\newcommand{\bela}[1]{\begin{equation}\label{#1}}
\newcommand{\ela}{\end{equation}}
\newcommand{\bear}[1]{\begin{array}{#1}}
\newcommand{\ear}{\end{array}}

\newcommand{\bolda}{\mbox{\boldmath$a$}}
\newcommand{\boldb}{\mbox{\boldmath$b$}}

\renewcommand{\r}{\mbox{\boldmath $r$}}
\newcommand{\F}{\mbox{\boldmath $F$}}

\newcommand{\bS}{\mbox{\boldmath $S$}}
\newcommand{\bT}{\mbox{\boldmath $T$}}

\newcommand{\bv}{\mbox{\boldmath $v$}}

\newcommand{\bomega}{\mbox{\boldmath $\omega$}}
\newcommand{\del}{\partial}

\newcommand{\Z}{\mathbbm{Z}}
\renewcommand{\P}{\mathbbm{P}}
\newcommand{\R}{\mathbbm{R}}

\theoremstyle{definition}
\newtheorem{theorem}{Theorem}[section]

\theoremstyle{definition}
\newtheorem{definition}[theorem]{Definition}

\newcommand{\Ft}{\tilde{\mbox{\boldmath$F$}}}
\newcommand{\br}{\mbox{\boldmath $r$}}
\newcommand{\btr}{\tilde{\mbox{\boldmath $r$}}}
\newcommand{\al}{\alpha}
\newcommand{\alb}{\bar{\alpha}}
\newcommand{\ab}{\bar{a}}
\newcommand{\bb}{\bar{b}}
\newcommand{\ub}{\bar{u}}
\newcommand{\vb}{\bar{v}}
\newcommand{\ka}{\kappa}
\newcommand{\kab}{\bar{\kappa}}
\newcommand{\C}{\mathcal{C}}
\newcommand{\T}{\bar{T}}
\newcommand{\f}{\bar{f}}
\newcommand{\g}{\bar{g}}
\newcommand{\Tb}{\bar{T}}
\newcommand{\vpt}{\tilde{\varphi}}
\newcommand{\psit}{\tilde{\psi}}
\newcommand{\bvpt}{\tilde{\mbox{\boldmath$\varphi$}}}
\newcommand{\bpsit}{\tilde{\mbox{\boldmath$\psi$}}}
\newcommand{\s}{\sigma}

\newcommand{\ba}{\begin{align}}
\newcommand{\bpm}{\begin{pmatrix}}
\newcommand{\epm}{\end{pmatrix}}
\newcommand{\beq}{\begin{equation}}
\newcommand{\ee}{\end{equation}}

\mathtoolsset{showonlyrefs,showmanualtags}

\title{Surface theory in discrete projective differential geometry. I.\ A canonical frame and an integrable discrete Demoulin system}

\author{W.K.\ Schief\\ \small
            School of Mathematics and Statistics\\ \small
            The University of New South Wales\\ \small 
            Sydney, NSW 2052\\ \small
           Australia
           \and
           A.\ Szereszewski\\ \small
           Institute of Theoretical Physics\\ \small
           Faculty of Physics\\ \small
           University of Warsaw\\ \small
          Poland}

\date{26th October 2017}
           
\begin{document}

\maketitle

\begin{abstract}
We present the first steps of a procedure which discretises surface theory in classical projective differential geometry in such a manner that underlying integrable structure is preserved. We propose a canonical frame in terms of which the associated projective Gauss-Weingarten and Gauss-Mainardi-Codazzi equations adopt compact forms. Based on a scaling symmetry which injects a parameter into the linear Gauss-Weingarten equations, we set down an algebraic classification scheme of discrete projective minimal surfaces which turns out to admit a geometric counterpart formulated in terms of discrete notions of Lie quadrics and their envelopes. In the case of discrete Demoulin surfaces, we derive a B\"acklund transformation for the underlying discrete Demoulin system and show how the latter may be formulated as a two-component generalisation of the integrable discrete Tzitz\'eica equation which has originally been derived in a different context. At the geometric level, this connection leads to the retrieval of the standard discretisation of affine spheres in affine differential geometry.
\end{abstract}

\section{Introduction}

Projective differential geometry (see \cite{OvsienkoTabachnikov2005, Eastwood2008} and references therein) has been demonstrated to be a rich source of surface geometries which are governed by integrable partial differential equations \cite{FerapontovSchief1999,Ferapontov2000}. In this context, the appropriate formalism has proven to be that of the ``American School'' founded by Wilczynski who, in fact, initiated projective differential geometry \cite{Wilczynski1907,Wilczynski1908,Wilczynski1909}. For instance, projective minimal and isothermal-asymptotic surfaces admit B\"acklund transformations which both act within these classes of surfaces and leave invariant the underlying projective Gauss-Mainardi-Codazzi equations \cite{RogersSchief2002}. The latter include periodic Toda lattice type systems and the stationary modified Nizhnik-Veselov-Novikov (mNVN) equation \cite{Ferapontov1999}. It is recalled that the mNVN equation constitutes a 2+1-dimensional integrable extension of the celebrated modified Korteweg-de Vries (mKdV) equation in which the two ``spatial'' variables appear on an equal footing \cite{Bogdanov1987}. 
Projective geometry also plays a central role in the geometric treatment of {\em discrete} integrable systems in that, for instance, the master Hirota (dKP), Miwa (dBKP) and dCKP equations are intimately related to classical incidence theorems of projective geometry  (see \cite{Schief2003,KingSchief2006} and references therein).

Wilczynski's formalism, which was adopted by Bol in the first two volumes of his monograph {\em Projektive Differentialgeometrie} \cite{Bol1950,Bol1954}, turns out to be custom-made in connection with not only the isolation of integrable structure but also the development of a canonical discrete analogue of projective differential geometry within the field of {\em discrete differential geometry} \cite{BobenkoSuris2008}. In \cite{McCarthySchief2018}, discrete analogues of a variety of classes of special surfaces in a three-dimensional real projective space $\P^3$ which feature prominently in Bol's second volume such as projective minimal surfaces, Q surfaces and complex surfaces have been proposed and analysed mainly in a geometric manner. By construction, these discretisations are natural in geometric terms but whether and in which sense this is reflected in the algebraic properties has only been touched upon. For instance, is it possible to introduce natural frames such as the Wilczynski frame (see, e.g., \cite{FerapontovSchief1999}) in terms of which the underlying discrete Gauss-Mainardi-Codazzi equations are compact and tractable? Does the discretisation scheme preserve any integrable structure which is present in the classical continuous setting? It is one aim of this paper to answer these questions in the affirmative.

We begin by deriving from first principles a canonical frame associated with discrete asymptotic nets. The latter have been used extensively \cite{BobenkoSuris2008} as discretisations of asymptotic nets on hyperbolic surfaces in discrete differential geometry. They naturally give rise to a discretisation of classical Lie quadrics \cite{HuhnenVenedeyRoerig2014,McCarthySchief2017} to which the discrete canonical frame is adapted. It is observed that this frame does not appear to have an analogue in the continuous setting. In terms of the canonical frame, the associated projective Gauss-Weingarten equations involve sparse matrices so that their integrability conditions lead to compact discrete projective Gauss-Mainardi-Codazzi equations. We then introduce a natural parameter-dependent scaling which leaves invariant all but one Gauss-Mainardi-Codazzi equations. Complete invariance leads to a constraint on the Gauss-Mainardi-Codazzi equations which coincides with that defining discrete projective minimal surfaces as proposed in \cite{McCarthySchief2018}. It is recalled that the analogous scaling symmetry is known to encode projective minimal surfaces in the classical setting \cite{RogersSchief2002}.

As in the continuous case \cite{Bol1954,FerapontovSchief1999,RogersSchief2002}, the form of the constrained Gauss-Mainardi-Codazzi equations lends itself to an algebraic classification of discrete projective minimal surfaces. In fact, we show that this approach defining different types of discrete projective minimal surfaces admits a geometric analogue based on the notion of discrete envelopes of the above-mentioned lattices of Lie quadrics. The geometric classification of classical projective minimal surfaces in terms of envelopes of Lie quadrics may be found in, for instance, \cite{Bol1954,Sasaki2005}. In order to embark on a study of the integrability properties of the Gauss-Mainardi-Codazzi equations underlying discrete projective minimal surfaces, we then focus on the case of discrete Demoulin surfaces. Based on the classical Pl\"ucker correspondence between lines in $\P^3$ and points in the Pl\"ucker quadric embedded in $\P^5$, we relate the canonical frame for discrete Demoulin surfaces to a frame of Wilczynski type, in terms of which the discrete Gauss-Mainardi-Codazzi equations adopt the form of a discrete analogue of the classical integrable Demoulin system \cite{Finikov1937}. The latter is then shown to be preserved by a discrete analogue of the B\"acklund transformation for the Demoulin system \cite{RogersSchief2002}.

As in the classical case, the discrete Demoulin system admits an even deeper reduction which turns out to be the discrete Tzitz\'eica equation proposed in \cite{Schief1999}. The latter has been demonstrated \cite{BobenkoSchief1999} to encode an integrable discretisation of the classical class of affine spheres in (centro-)affine differential geometry (see, e.g., \cite{Schief2000} and references therein). Remarkably, the (scaled) discrete Wilczynski frame corresponding to the Tzitz\'eica reduction turns out to capture nothing but standard discrete affine spheres subject to projective transformations. Hence, within the established framework of discrete differential geometry, a link between the discretisation technique analysed here in the setting of discrete projective differential geometry and the {\em a priori} unrelated discretisation procedure for affine spheres set down in \cite{BobenkoSchief1999} has been found.

\section{Surface theory in projective differential geometry}

In order to set the discrete theory in context, we here briefly recall relevant classical facts. Thus, in the classical theory \cite{Bol1954}, one is concerned with surfaces $\Sigma$ in a three-dimensional projective space $\mathbb{P}^3$ represented in terms of homogeneous coordinates by \mbox{$\r:\mathbb{R}^2\rightarrow\mathbb{R}^4$},
where $(x,y)\in\mathbb{R}^2$ are taken to be {\em asymptotic coordinates} on $\Sigma$. Since we confine ourselves to hyperbolic surfaces, the asymptotic coordinates are real. Then, it is well known \cite{Bol1954,FerapontovSchief1999,RogersSchief2002} that one may choose particular homogeneous coordinates, known as the {\em Wilczynski lift}  \cite{Wilczynski1907,Wilczynski1908,Wilczynski1909}, such that $\r$ satisfies the {\em projective Gauss-Weingarten equations}
    \begin{equation*}
        \mbox{\boldmath$r$}_{xx}=p\mbox{\boldmath$r$}_y+\frac{1}{2}\left(\frac{q_{xx}}{q}-\frac{q_x^2}{2q^2} + \frac{\beta}{q^2}-p_y\right)\mbox{\boldmath$r$},\quad
        \mbox{\boldmath$r$}_{yy}=q\mbox{\boldmath$r$}_x+\frac{1}{2}\left(\frac{p_{yy}}{p}-\frac{p_y^2}{2p^2} + \frac{\alpha}{p^2}-q_x\right)\mbox{\boldmath$r$}.
    \end{equation*}
It is noted that the Wilczynski lift is unique up to a group of transformations which involves simultaneously reparametrising the asymptotic lines and scaling the homogeneous coordinates $\r$ (see, e.g., \cite{FerapontovSchief1999}). The compatibility condition $\mbox{\boldmath$r$}_{xxyy}=\mbox{\boldmath$r$}_{yyxx}$ leads to the {\em projective Gauss-Mainardi-Codazzi equations}
\bela{gmc}
    \begin{aligned}
        (\ln p)_{xy}&=pq+\frac{\mathcal{A}}{p},\val&
        \mathcal{A}_y&=-p\left(\frac{\alpha}{p^2}\right)_x,\quad\frac{\alpha_y}{p}=\frac{\beta_x}{q}\\
        (\ln q)_{xy}&=pq+\frac{\mathcal{B}}{q},\val&
        \mathcal{B}_x&=-q\left(\frac{\beta}{q^2}\right)_y
    \end{aligned}
\ela
with \eqref{gmc}$_{1,4}$ being regarded as definitions of the functions $\mathcal{A}$ and $\mathcal{B}$. In the following, we exclude ruled surfaces so that $p\neq0$ and $q\neq0$.

\subsection{Algebraic classification of projective minimal surfaces}
    \begin{definition}
        A surface in $\mathbb{P}^3$
        is said to be {\em projective minimal} if it is critical for the
        area functional
        $\iint pq\,dxdy.$
    \end{definition}
The derivation of the associated Euler-Lagrange equations may be found in \cite{Thomsen1925}.
    \begin{theorem}\label{elcond}
    A surface in $\mathbb{P}^3$ is projective minimal if
    and only if
        \begin{equation}\label{eulerlagrange}
            \frac{\alpha_y}{p}=\frac{\beta_x}{q}=0,
        \end{equation}
that is, $\alpha=\alpha(x)$ or, equivalently (by virtue of \eqref{gmc}$_3$), $\beta=\beta(y)$.
 \end{theorem}

\begin{definition}
A projective minimal surface $\Sigma$ is said to be
    \begin{itemize}
        \item[(i)]
            {\em generic} if $\alpha\neq 0$ and $\beta\neq 0$,\vspace{0.5mm}
        \item[(ii)]
            of {\em Godeaux-Rozet} type if $\alpha\neq0$ and $\beta=0$ or
            $\alpha=0$ and $\beta\neq0$,\vspace{0.5mm}
        \item[(iii)]
            of {\em Demoulin type} if $\alpha=\beta=0$. If, in addition,
            $p=q$ then $\Sigma$ is said to be of {\em Tzitz\'{e}ica} type.
    \end{itemize}
\end{definition}
 
It is noted that, in this situation, the above-mentioned group of transformations which acts within the class of Wilczynski frames may be exploited to normalise $\alpha$ and $\beta$ to be one of $-1,\,1$ or $0$. This normalisation corresponds to canonical forms of the integrable system \eqref{gmc}-\eqref{eulerlagrange} underlying projective minimal surfaces \cite{FerapontovSchief1999}.

\subsection{Geometric classification of projective minimal surfaces}
Projective minimal surfaces may also be classified geometrically based on the notion of Lie quadrics and their envelopes. A {\em Lie quadric} is a privileged member of the three-parameter family of quadrics which has second-order contact with a surface \cite{Bol1954} at a given point. In the current context, the key observation is the following \cite{Bol1954,Ferapontov2000}. 
    \begin{theorem}
        The Lie quadric $Q$ (at a point $\r$) of a surface $\Sigma$ admits the parametrisation
            \begin{equation*}
                Q=\btr^{12}+\mu\btr^1+\nu\btr^2+\mu\nu\btr,
            \end{equation*}
        where $\mu$ and $\nu$ parametrise the two families of generators of $Q$ and
        $\{\btr,\btr^1,\btr^2,\btr^{12}\}$ is the {\em Wilczynski frame} given by
\bela{C2}
 \begin{gathered}
  \btr = \r,\quad \btr^1 = \r_x  - \frac{1}{2}\frac{q_x}{q}\r,\quad  \btr^2 = \r_y  - \frac{1}{2}\frac{p_y}{p}\r\as
\btr^{12} = \r_{xy} - \frac{1}{2}\frac{p_y}{p}\r_x - \frac{1}{2}\frac{q_x}{q}\r_y + \left(\frac{1}{4}\frac{p_yq_x}{pq} - \frac{1}{2}pq\right)\r .
 \end{gathered} 
\ela
    \end{theorem}
In the above, for brevity, we do not distinguish notationally between a Lie quadric in
$\mathbb{P}^3$ and its representation in the space of
homogeneous coordinates $\mathbb{R}^4$.
It is also observed that the lines $(\btr,\btr^1)$ and
$(\btr,\btr^2)$ are tangent to the surface $\Sigma$, while the line
$(\btr,\btr^{12})$, known as the  {\em first directrix
of Wilczynski}, is transversal to $\Sigma$ and plays
the role of a projective normal.
    \begin{definition}
        A surface $\Omega$ parametrised in terms of homogeneous coordinates by
        $\bomega:\mathbb{R}^2\rightarrow\mathbb{R}^4$ is an {\em envelope}
        of the two-parameter family of Lie quadrics $\{Q(x,y)\}$ associated with a surface $\Sigma$ if $\bomega(x,y)\in Q(x,y)$ such that
        $\Omega$ touches $Q(x,y)$ at $\bomega(x,y)$.
    \end{definition}
We note that, in particular, $\Sigma$ is itself an envelope of
$\{Q\}$. Generically, there exist four additional envelopes as
stated below \cite{Bol1954}.
    \begin{theorem}\label{envel}
        If $\alpha,\,\beta\geq0$ then the Lie quadrics $\{Q\}$ possess
        four real additional envelopes
            \bela{envelope}
                \bomega =\btr^{12}+\epsilon_1\hat{\mu}\btr^1+\epsilon_2\hat{\nu}\btr^2+\epsilon_1\epsilon_2\hat{\mu}\hat{\nu}\btr,\qquad
                \hat{\mu}=\sqrt{\frac{\alpha}{2p^2}},\quad\hat{\nu}=\sqrt{\frac{\beta}{2q^2}},
            \ela
     where $\epsilon_i=\pm1$. These are distinct if $\alpha,\beta \neq 0$.
    \end{theorem}

The above expressions for $\hat{\mu}$
and $\hat{\nu}$ reveal that whether $\alpha$ and $\beta$ vanish or
not is related to the number of distinct envelopes. Accordingly, the geometric
interpretation of the algebraic classification recalled in the preceding is as follows. A projective minimal
surface $\Sigma$ is
    \begin{itemize}
        \item[(i)]
            generic if the set of Lie quadrics $\{Q\}$ has four
            distinct additional envelopes,\vspace{0.5mm}
        \item[(ii)]
            of Godeaux-Rozet type if $\{Q\}$ has two distinct additional
            envelopes,\vspace{0.5mm}
        \item[(iii)]
            of Demoulin type if $\{Q\}$ has one additional
            envelope.
    \end{itemize}

We now state a classical theorem which lies at the heart of the geometric definition and analysis of discrete projective minimal surfaces. In order to do so, we adopt a definition proposed in \cite{McCarthySchief2017}.

\begin{definition}
  A surface is termed a {\em PMQ surface} if its asymptotic lines correspond to the asymptotic lines on at least one associated envelope.
\end{definition}

\begin{theorem}
  The class of PMQ surfaces consists of projective minimal (PM) and Q surfaces.
\end{theorem}

The above key theorem may be found in \cite{Bol1954,Sasaki2005}. For the definition of {\em Q surfaces}, we refer to the monograph \cite{Bol1954} or \cite{McCarthySchief2018}. However, the discrete analogue of Q surfaces is defined in Section 5(a).

\section{Discrete surfaces in \mbox{\boldmath $\P^3$}}

We are now concerned with discrete surfaces $\Sigma$ in a real projective space $\P^3$, that is, lattices of $\Z^2$ combinatorics in $\P^3$ which are represented by homogeneous coordinates $\r :\Z^2\rightarrow\R^4$. If we indicate increments and decrements of the discrete independent variables $n_k$, $k=1,2$ by subscripts $k$ and $\bar{k}$ respectively then any quadrilateral of a discrete surface $\Sigma$ is denoted by $[\r,\r_1,\r_2,\r_{12}]$, while the 5 vertices of any star are given by the central vertex $\r$ and its nearest neighbours $\r_1,\r_2$ and $\r_{\bar{1}},\r_{\bar{2}}$. Here, and in the following, we suppress the arguments in $\r(n_1,n_2)$ and simply write $\r$, provided that this does not give rise to ambiguity. Furthermore, we focus on the standard (integrable) discretisation of (hyperbolic) surfaces parametrised in terms of asymptotic coordinates \cite{BobenkoSuris2008}.

\begin{definition}
A discrete surface $\Sigma$ in $\P^3$ represented by a map $\r:\Z^2\rightarrow\R^4$ is termed a {\em discrete asymptotic net} if the stars of $\Sigma$ are planar, that is, if any vertex $\r$ and its four nearest neighbours $\r_1,\r_2, \r_{\bar{1}},\r_{\bar{2}}$ (regarded as points in $\P^3$) are coplanar.
\end{definition}

It is observed that, in algebraic terms, the conditions for a surface to constitute a discrete asymptotic net may be formulated as 
%
  $|\r,\r_1,\r_{11},\r_{12}| = 0$ and $|\r,\r_2,\r_{22},\r_{12}| = 0$.
%
In the following, we also assume that any discrete asymptotic net is generic in the sense that its quadrilaterals are non-planar. In this case, it is easy to verify \cite{HuhnenVenedeyRoerig2014} that any quadrilateral gives rise to a one-parameter family (pencil) of quadrics which pass through its edges. We say that any quadric of this family is {\em associated} with the quadrilateral. Accordingly, the following definition proposed in \cite{McCarthySchief2017,McCarthySchief2018} is natural.

\begin{definition}
Any two quadrics associated with two neighbouring quadrilaterals of a discrete asymptotic net have the {\em $\C^1$ property} if the tangent planes of the two quadrics coincide at each point of the common edge of the quadrilaterals. A {\em lattice of Lie quadrics} $\{Q\}$ is a set of quadrics associated with the quadrilaterals of a discrete asymptotic net such that any two neighbouring quadrics have the $\C^1$ property.
\end{definition}

It turns out that any discrete asymptotic net gives rise to a one-parameter family of lattices of Lie quadrics \cite{McCarthySchief2017,McCarthySchief2018}. Thus, if $Q$ is a quadric associated with a quadrilateral $\Box = [\r,\r_1,\r_2,\r_{12}]$ then the $\C^1$ condition uniquely determines quadrics $Q_1$ and $Q_2$ associated with the quadrilaterals $\Box_1$ and $\Box_2$ respectively. Accordingly, there exist two quadrics $Q_{12}$ and $Q_{21}$ associated with the quadrilateral $\Box_{12}$ which are uniquely determined by the $\C^1$ condition with respect to the quadrics $Q_1$ and $Q_2$. Remarkably, the two quadrics $Q_{12}$ and $Q_{21}$ coincide \cite{HuhnenVenedeyRoerig2014}. This is summarised in the following theorem.

\begin{theorem}
A lattice of Lie quadrics associated with a discrete asymptotic net is uniquely determined by prescribing the quadric associated with one quadrilateral.
\end{theorem}

\subsection{A canonical frame}

In order to derive a compact form of Gauss-Weingarten-type equations for discrete asymptotic nets $\Sigma$, we introduce the frame
\bela{E2}
 \F = 
   \left(\begin{array}{l}
             \br\\ \br^1\\ \br^2\\ \br^{12}
    \end{array}
   \right) = 
   \left(\begin{array}{c}
             \br\\ \alpha\br_1\\ \alb\br_2\\ \gamma\br_{12} 
    \end{array}
   \right),
\ela 
wherein the functions $\alpha,\alb$ and $\gamma$ are to be determined. The planarity of stars characterising discrete asymptotic nets then implies that neighbouring frames are related by a linear system of the form
\bela{F1_1}
   \left(\begin{array}{l}
             \br\\ \br^1\\ \br^2\\ \br^{12} 
    \end{array}
    \right)_1 = 
    \left(\bear{cccc}
      0   &  1/\alpha  &  0      &  0  \\
    \beta &      u     &  0      &  b  \\  
      0   &      0     &  0      &  \alb_1/\gamma\\
      0   &      a     & \delta  &  v  \\
   \ear\right)
   \left(\begin{array}{l}
             \br\\ \br^1\\ \br^2\\ \br^{12} 
    \end{array}
    \right),
    \quad
    \left(\begin{array}{l}
             \br\\ \br^1\\ \br^2\\ \br^{12} 
    \end{array}
    \right)_2 = 
    \left(\bear{cccc}
         0   &      0         &   1/\alb  &    0       \\
         0   &      0         &     0    &  \alpha_2/\gamma\\[1mm]
   \bar{\beta} &      0         &    \ub   &  \bb  \\  
         0   &  \bar{\delta}  &    \ab   &  \vb   
   \ear\right)
   \left(\begin{array}{l}
             \br\\ \br^1\\ \br^2\\ \br^{12} 
    \end{array}
    \right) .   
\ela
These frame equations for discrete asymptotic nets may be simplified by appropriately choosing the homogeneous coordinates. Thus, we first scale the frame vectors in such a way that the determinant of $\F$ is constant. Accordingly, the determinants of the
matrices in \eqref{F1_1} are unity so that
\bela{det_conds}
 \frac{\alb_1\beta\delta}{\alpha\gamma}=1, \quad \frac{\alpha_2\bar{\beta}\bar{\delta}}{\alb\gamma} =1.
\ela

The next simplification is related to the fact that the one-parameter ($p$) family of quadrics associated with any quadrilateral $[\r,\r_1,\r_2,\r_{12}]$ may be represented by
\bela{E4}
  Q = p\r^{12} + \mu\r^1 + \nu\r^2 + \mu\nu\r,
\ela
where $\mu$ and $\nu$ parametrise any quadric $Q$ in this family for fixed $p$.  For brevity, here and in the following, we use the same symbol $Q$ for a quadric and its representation in terms of homogeneous coordinates. It is noted that the coordinate lines $\mu=\mbox{const}$ and $\nu=\mbox{const}$ make up the two families of generators of $Q$. It is also emphasised that, for any scalar $\varkappa$,  $Q$ and $\varkappa Q$ represent the same quadric. For any given lattice of Lie quadrics $\{Q\}$, we may always scale the frame vectors in such a manner that $p=1$. Hence, for instance, the quadrics $Q$ and $Q_1$ of any neighbouring quadrilaterals $\Box = [\r,\r_1,\r_2,\r_{12}]$ and $\Box_1$ are parametrised by
\bela{E5}
  Q = \br^{12} + \mu \br^1 + \nu \br^2 + \mu\nu\br,\quad
  Q_1 = \br^{12}_1 + \mu_1 \br^1_1 + \nu_1 \br^2_1 + \mu_1\nu_1\br_1
\ela
which shows that the common edge $[\r_1,\r_{12}]$ of the quadrilaterals $\Box$ and $\Box_1$ is represented by \mbox{$\nu=0$} and $\nu_1=\infty$ respectively. Thus, the generators $\mu=\mbox{const}$ and $\mu_1=\mbox{const}$ of $Q$ and $Q_1$ respectively meet at the point
\bela{E6}
  P \sim \r^{12} + \mu\r^1 \sim \r^2_1 + \mu_1\r_1 = \frac{\alb_1}{\gamma}\left(\br^{12}+\frac{\gamma}{\alpha\alb_{1}}\mu_1\br^1\right)
\ela
by virtue of the frame equations \eqref{F1_1} so that the labels $\mu$ and $\mu_1$ of two generators which meet at the common edge are related by
\begin{equation}
   \mu =  \frac{\gamma}{\alpha\alb_{1}}\mu_1.   \label{mumu1}
 \end{equation}  
Furthermore, the $\C^1$ property at any point $P$ on the common edge may be formulated as
 \begin{equation}
   \left| \frac{\partial}{\partial\nu}{Q|}_P,\,\br^1,\, \br^{12},\, \frac{\partial}{\partial\hat{\nu}_1} {\hat{Q}_1|}_P\right| = 0,   \label{C1cond}
 \end{equation}
where $\hat{Q}_1=\hat{\nu}_1 Q_1$ and $\hat{\nu}_1=1/\nu_1$, leading to the relation 
\bela{E7}
 \mu_1\beta=\mu\delta.
\ela
Similarly, consideration of any two neighbouring quadrilaterals $\Box$ and $\Box_2$ results in the analogous relations
\bela{E8}
  \nu =  \frac{\gamma}{\alb\alpha_{2}}\nu_2,\quad \nu_2\bar{\beta}=\nu\bar{\delta}.
\ela
Elimination of $\mu,\mu_1$ and $\nu,\nu_2$ from \eqref{mumu1}, \eqref{E7} and \eqref{E8} respectively therefore produces
 \begin{equation}
  \delta=\frac{\alpha\bar{\alpha}_1}{\gamma} \beta, \quad \bar{\delta}=\frac{\bar{\alpha}\alpha_2}{\gamma} \bar{\beta}.      \label{deltas}
 \end{equation}

Finally, the remaining two degrees of freedom in the choice of homogeneous coordinates may be exploited to guarantee that the labels of any two generators of neighbouring quadrics meeting at the common edge of the two corresponding quadrilaterals are the same, that is,
\begin{equation}
    \mu_1=\mu, \quad \nu_2=\nu.   \label{munu_conds}
 \end{equation}
Comparison with \eqref{mumu1} and \eqref{E8}$_1$ then shows that
\begin{equation}
   \alb_1=\frac{\gamma}{\al}, \quad \alpha_2=\frac{\gamma}{\alb}.  \label{alphas}
\end{equation}
The latter may be regarded as a definition of $\gamma$ together with the constraint
\bela{E9}
   \al\alb_1 = \alb\al_2.
\end{equation}
In Section 3(b), it is demonstrated that this constraint is implied by the compatibility condition for the frame equations \eqref{F1_1} so that it may be set aside in the current context. We now observe that the relations \eqref{deltas} reduce to 
 \begin{equation}\label{E10}
  \delta = \beta, \quad \bar{\delta} = \bar{\beta}
 \end{equation}
so that the determinant conditions \eqref{det_conds} simplify to $\beta^2=\al^2$ and $\bar{\beta}^2=\alb^2$. Without loss of generality, we may therefore set
\begin{equation}\label{E11}
    \beta = -\al, \quad \bar{\beta} = -\alb.
\end{equation} 
The preceding analysis is summarised in the following theorem.

\begin{theorem}
  A discrete asymptotic net $\Sigma$ in $\P^3$ admits a {\em canonical frame} $\F$ given by \eqref{E2}, wherein $\gamma$ is defined by either of the relations \eqref{alphas}, such that the {\em discrete projective Gauss-Weingarten equations} adopt the form 
\bela{Feq1}
 \bear{rcl}
   \F_1&=&\left(\begin{array}{l}
             \br\\ \br^1\\ \br^2\\ \br^{12} 
    \end{array}
    \right)_1 = 
    \left(\bear{cccc}
      0   &  1/\alpha  &  0      &  0  \\
    -\al  &      u     &  0      &  b  \\  
      0   &      0     &  0      &  1/\al\\
      0   &      a     & -\al    &  v  \\
   \ear\right)
\left(\begin{array}{l}
             \br\\ \br^1\\ \br^2\\ \br^{12} 
    \end{array}
    \right)=L \F \\[10mm]
  \F_2&=&\left(\begin{array}{l}
             \br\\ \br^1\\ \br^2\\ \br^{12} 
    \end{array}
    \right)_2 = 
    \left(\bear{cccc}
         0   &      0         &   1/\alb  &    0       \\
         0   &      0         &      0    &  1/\alb\\[1mm]
      -\alb  &      0         &     \ub   &  \bb  \\  
         0   &    -\alb       &     \ab   &  \vb   
   \ear\right)
\left(\begin{array}{l}
             \br\\ \br^1\\ \br^2\\ \br^{12} 
    \end{array}
    \right)=M \F
  \ear
\ela
and the quadrics
\bela{E12}
  Q = \br^{12} + \mu \br^1 + \nu \br^2 + \mu\nu\br
\ela
constitute a lattice of Lie quadrics $\{Q\}$ associated with $\Sigma$.
\end{theorem}

In the following, we assume that no triple of consecutive vertices of any coordinate polygon are collinear. This is equivalent to demanding that the functions $a,\ab$ and $b,\bb$ do not vanish and may be regarded as a discrete analogue of the exclusion of ruled surfaces.

\subsection{Discrete Gauss-Mainardi-Codazzi equations}

The compatibility condition $\F_{12}=\F_{21}$ for the discrete Gauss-Weingarten equations \eqref{Feq1} leads to the nonlinear system of difference equations $L_2M = M_1L$. These assume a compact form if one adopts the change of dependent variables
\bela{FFggb}
 f = \frac{u+v}{2},\quad\f = \frac{\ub+\vb}{2},\quad
 g = \frac{u-v}{2},\quad\g = \frac{\ub-\vb}{2}   
\ela
together with the definitions
\bela{E14}
 T = ab + g^2, \quad\Tb = \ab\bb + \g^2.
\ela
The following theorem is then directly verified.

\begin{theorem}
The {\em discrete projective Gauss-Mainardi-Codazzi equations} associated with the Gauss-Weingarten equations \eqref{Feq1} for a discrete asymptotic net are given by
\bela{E15}
\begin{aligned}
     \al_2 &= w\al, \val       &                    
    wf_2 &=  f - \frac{a}{\alb} \g , \val    &     
    wg_2 &= -g + \frac{a}{\alb} \f, \val    &     
     w b_2 &= -\frac{a}{\alb^2}\\   
     \alb_1 &=  w\alb,\val &
     w\f_1 &=  \f - \frac{\ab}{\al} g,\val &
     w\g_1 &= -\g + \frac{\ab}{\al} f,\val &
     w \bb_1 &= -\frac{\ab}{\al^2}
 \end{aligned} 
\ela
together with
\begin{equation}
   \ab \alb \Delta_2T = a \al \Delta_1\Tb,  \label{Teqn}
\end{equation}
where $\Delta_i h = h_i - h$ for any function $h$ and $w$ is defined by 
\begin{equation}
    \al\alb(w^2-1)+a\ab=0   \label{weqn}.
 \end{equation}
The constraint \eqref{E9} is implied by \eqref{E15}$_{1,5}$.
\end{theorem}

It turns out that the quantities $T$ and $\Tb$ have a distinct geometric meaning. Thus, if two quadrics $Q$ and $Q_1$ associated with two neighbouring quadrilaterals $\Box$ and $\Box_1$ have the $\C^1$ property then we may regard these to be part of the lattice of Lie quadrics $\{Q\}$ to which the canonical frame is adapted. The points which the two quadrics have in common are then determined by $Q\sim Q_1$, where $Q$ and $Q_1$ are given by \eqref{E5}. This leads to the set of equations
\bela{E16}
   \mu  = \mu_1,\quad
   \frac{\nu}{\nu_1}\left(b\mu_1+v\right)+\frac{\nu}{\al}+\frac{\al}{\nu_1} = 0,\quad
   \frac{\nu}{\nu_1}\left(u\mu_1+a\right)+\frac{\nu}{\al}\mu_1+\frac{\al}{\nu_1}\mu = 0. 
\ela
The first equation implies that the points of intersection are points common to generators \mbox{$\mu=\mbox{const}$} and $\mu_1=\mbox{const}$ of $Q$ and $Q_1$ respectively which meet at the common edge of $\Box$ and $\Box_1$. The case $\nu=0$, $\nu_1=\infty$ corresponds to the (extended) edge $[\r_1,\r_{12}]$ which belongs to both quadrics $Q$ and $Q_1$. If $\nu\neq0$ then the remaining two equations reduce to
\bela{E17}
  b\mu^2 -(u-v)\mu-a = 0,\quad
  \nu_1 = -\al\left(\frac{\al}{\nu}+b\mu+v\right).
\ela
Accordingly, in addition to the common edge, the two quadrics intersect in common generators $\mu=\mu_1=\mbox{const}$ determined by the quadratic equation \eqref{E17}$_1$. The nature of its solutions depends on the sign of the discriminant
 \begin{equation}\label{E18}
    4ab + (u-v)^2 = 4T.
 \end{equation}
This is summarised below.

\begin{theorem}
In terms of the adapted canonical frame, any two neighbouring quadrics $Q$ and $Q_1$ of a lattice of Lie quadrics $\{Q\}$ meet in the extended common edge and two additional common generators $\mu=\mu_1=\mbox{const}$ if $T\geq0$ with $4T$ playing the role of the discriminant of the quadratic equation
\bela{E19}
  b\mu^2 - 2g\mu - a = 0.
\ela
If $T=0$ then the two common generators coincide and the quadrics $Q$ and $Q_1$ touch along the common generator. If $T<0$ then the two quadrics only intersect in the extended common edge. The same statements apply {\em mutatis mutandis} to any neighbouring quadrics $Q$ and $Q_2$ with $4\bar{T}$ being the discriminant of the quadratic equation
\bela{E20}
  \bb\nu^2 - 2\g\nu - \ab = 0.
\ela
\end{theorem}

\begin{proof}
If $T=0$ then the condition for the two tangent planes of $Q$ and $Q_1$ to coincide at a point of the common generator is
\bela{E21}
  \left|Q,\frac{\del Q}{\del\nu},\frac{\del Q}{\del\mu},\frac{\del Q_1}{\del\mu_1}\right|=0\quad\Leftrightarrow\quad\frac{\nu_1}{\alpha} = -\frac{\alpha}{\nu} + b\mu - u.
\ela
The latter coincides with \eqref{E17}$_2$ since $u-v=2b\mu$ for $T=0$.
\end{proof}

\section{Discrete projective minimal surfaces}

In analogy with the classical case, we now propose an algebraic definition of discrete projective minimal surface based on a scaling invariance of the Gauss-Mainardi-Codazzi equations. Thus, we first observe that the Gauss-Mainardi-Codazzi equations \eqref{E15} are invariant under the scaling
\bela{E23}
  (a,b,g)\rightarrow\lambda(a,b,g),\quad (\ab,\bb,\g)\rightarrow\lambda^{-1}(\ab,\bb,\g)
\ela
with all other quantities being unchanged. The remaining Gauss-Mainardi-Codazzi equation \eqref{Teqn} becomes
 \begin{equation}
  \lambda^2 \ab \alb \Delta_2T = a \al\Delta_1\T  \label{scTeqn}
 \end{equation}
so that complete invariance is achieved by imposing the equivalent constraints
 \begin{equation}\label{E24}
   \Delta_2 T= 0, \quad \Delta_1 \Tb =0.
 \end{equation}
Hence, the following definition is natural and coincides with that proposed in \cite{McCarthySchief2018}. It constitutes a discrete analogue of the classical algebraic classification outlined in Section 2.

\begin{definition}
A discrete asymptotic net $\Sigma$ is a {\em discrete projective minimal surface} if there exists a lattice of Lie quadrics such that, in terms of the adapted canonical frame, $ \Delta_2 T= 0$ or, equivalently,  $\Delta_1 \Tb= 0$. In particular,
 \begin{itemize}
 \item[(i)] if $T=0$ or $\T = 0$ then $\Sigma$ is termed a {\em discrete Godeaux-Rozet surface},\vspace{0.5mm}
 \item[(ii)] if $T=\Tb=0$ then $\Sigma$ is termed a {\em discrete Demoulin surface}. 
 \end{itemize}
\end{definition}

The constrained Gauss-Mainardi-Codazzi equations \eqref{E15}, \eqref{E24} constitute the discrete analogue of the Euler-Lagrange equations for projective minimal surfaces. Since, for any choice of the sign of $w$ in \eqref{weqn} (which is geometrically irrelevant), this system of discrete equations constitutes a map of the form
\bela{E25}
 (\bv,\bar{\bv}) \mapsto (\bv_2,\bar{\bv}_1),\qquad \bv = (\alpha,a,b,f,g),\quad \bar{\bv} = (\alb,\ab,\bb,\f,\g),
\ela
a discrete projective minimal surface together with the corresponding lattice of Lie quadrics is uniquely determined by the Cauchy data
%
  $\bv(n_1,n_2=0)$ and $\bar{\bv}(n_1=0,n_2)$
%
up to projective transformations.

By construction, the Gauss-Weingarten equations \eqref{Feq1} subject to the scaling \eqref{E23}, that is,
\bela{GW}
 \bear{rcl}
   \left(\begin{array}{l}
             \br\\ \br^1\\ \br^2\\ \br^{12} 
    \end{array}
    \right)_1 &=& 
    \left(\bear{cccc}
      0   &  1/\alpha  &  0      &  0  \\
    -\al  &      f + \lambda g     &  0      &  \lambda b  \\  
      0   &      0     &  0      &  1/\al\\
      0   &      \lambda a     & -\al    &  f - \lambda g  \\
   \ear\right)
\left(\begin{array}{l}
             \br\\ \br^1\\ \br^2\\ \br^{12} 
    \end{array}
    \right) \\[10mm]
  \left(\begin{array}{l}
             \br\\ \br^1\\ \br^2\\ \br^{12} 
    \end{array}
    \right)_2 &=& 
    \left(\bear{cccc}
         0   &      0         &   1/\alb  &    0       \\
         0   &      0         &      0    &  1/\alb\\[1mm]
      -\alb  &      0         &     \f + \g/\lambda   &  \bb/\lambda  \\[1mm]  
         0   &    -\alb       &     \ab/\lambda   &  \f - \g/\lambda   
   \ear\right)
\left(\begin{array}{l}
             \br\\ \br^1\\ \br^2\\ \br^{12} 
    \end{array}
    \right)
  \ear
\ela
are compatible if and only if $\bv$ and $\bar{\bv}$ obey the Gauss-Mainardi-Codazzi equations for discrete projective minimal surfaces. In general, the significance of this $\lambda$-dependent linear system in the context of integrability will be discussed elsewhere. However, in Section 6, we establish a connection with the Lax pair for the nonlinear system underlying discrete Demoulin surfaces.

The above linear system also encodes two linear 2$\times$2 systems which are intimately related to the geometric characterisation of discrete projective minimal surfaces discussed in the next section. Indeed, if we make the formal substitution
%
  $(\r,\r^1,\r^2,\r^{12}) \rightarrow  \lambda^{n_1}(\rho_\infty,\rho^1_\infty,\rho^2_\infty,\rho^{12}_\infty)$
%
then, in the limit $\lambda\rightarrow\infty$, system \eqref{GW} reduces to $\rho_\infty=\rho^2_\infty=0$ and
\bela{E28}
  \left(\bear{c}\rho^1_\infty\\[1mm] \rho^{12}_\infty\ear\right)_1 = \left(\bear{cc}g & b\\ a &- g\ear\right)\left(\bear{c}\rho^1_\infty\\[1mm] \rho^{12}_\infty\ear\right),\quad \left(\bear{c}\rho^1_\infty\\[1mm] \rho^{12}_\infty\ear\right)_2 = \left(\bear{cc} 0 & 1/\alb\\ -\alb & \f\ear\right)\left(\bear{c}\rho^1_\infty\\[1mm] \rho^{12}_\infty\ear\right).
\ela
The determinants of the above matrices are $-T$ and $1$ respectively so that the compatibility condition associated with the above system gives rise to $\Delta_2 T = 0$. In fact, it is readily verified that the above system is compatible modulo the Gauss-Mainardi-Codazzi equations subject to this minimality condition. Similarly, the formal substitution 
%
  $(\r,\r^1,\r^2,\r^{12}) \rightarrow  \lambda^{-n_2}(\rho_0,\rho^1_0,\rho^2_0,\rho^{12}_0)$
%
leads to $\rho_0=\rho^1_0=0$ and
\bela{E30}
  \left(\bear{c}\rho^2_0\\[1mm] \rho^{12}_0\ear\right)_1 = \left(\bear{cc}0 & 1/\alpha\\ -\alpha & f\ear\right)\left(\bear{c}\rho^2_0\\[1mm] \rho^{12}_0\ear\right),\quad \left(\bear{c}\rho^2_0\\[1mm] \rho^{12}_0\ear\right)_2 = \left(\bear{cc} \g & \bb\\ \ab & -\g\ear\right)\left(\bear{c}\rho^2_0\\[1mm] \rho^{12}_0\ear\right)
\ela
in the case $\lambda=0$. Once again, modulo the Gauss-Mainardi-Codazzi equations, the associated compatibility condition yields the minimality condition $\Delta_1\T=0$. 

\section{Envelopes of lattices of Lie quadrics}

It turns out that discrete projective minimal surfaces may be characterised geometrically in terms of envelopes of the associated lattice of Lie quadrics \cite{McCarthySchief2018}. 

\begin{definition}
An {\em envelope} $\Omega$ of a lattice of Lie quadrics $\{Q\}$ associated with a discrete asymptotic net $\Sigma$ is a combinatorially dual discrete asymptotic net represented by a map $\bomega : \Z^2\rightarrow\R^4$ such that the star of $\Omega$ centred at any vertex $\bomega$ touches the corresponding quadric $Q$ at $\bomega\in Q$.
\end{definition}

Given a discrete asymptotic net $\Sigma$ and an associated lattice of Lie quadrics $\{Q\}$ with adapted canonical frame $\F$, let $\Omega$ be a combinatorially dual discrete net represented by a map $\bomega : \Z^2\rightarrow \R^4$ such that each vertex $\bomega$ lies on the corresponding quadric $Q$, that is,
\bela{E30a}
 \begin{split}
   \bomega(n_1,n_2) = \r^{12}(n_1,n_2) &+ \mu(n_1,n_2)\r^1(n_1,n_2)\\ &+ \nu(n_1,n_2)\r^2(n_1,n_2) + \mu(n_1,n_2)\nu(n_1,n_2)\r(n_1,n_2)
 \end{split}
\ela
for some lattice of parameters $\mu=\mu(n_1,n_2)$ and $\nu=\nu(n_1,n_2)$. Once again, we suppress the arguments of $\mu$ and $\nu$ whenever the context reveals whether $\mu$ and $\nu$ are parameters which parametrise the generators of a quadric $Q$ or $\mu$ and $\nu$ refer to specific points on $Q$. Thus, for instance,
\bela{E31}
  \bomega = \r^{12} + \mu\r^1 + \nu\r^2 + \mu\nu\r,\quad
  \bomega_1 = \r^{12}_1 + \mu_1\r^1_1 + \nu_1\r^2_1 + \mu_1\nu_1\r_1
\ela
designates two points on neighbouring quadrics $Q$ and $Q_1$ which may be vertices of an envelope $\Omega$. Now, $\Omega$ constitutes an envelope of the lattice of Lie quadrics if the {\em tangency condition} is satisfied, that is, if any edges $[\bomega,\bomega_1]$ and $[\bomega,\bomega_2]$ are tangent to the corresponding pairs of quadrics $Q,Q_1$ and $Q,Q_2$ respectively. For instance, the condition that an edge $[\bomega,\bomega_1]$ touches the two corresponding quadrics $Q$ and $Q_1$ at $\bomega$ and $\bomega_1$ respectively may be formulated as
\begin{equation}
  \left|\bomega, \bomega_1, \frac{\partial}{\partial\mu}{Q|}_{\bomega}, \frac{\partial}{\partial\nu} 
   {Q|}_{\bomega}\right| = 0, \quad
  \left| \bomega, \bomega_1, \frac{\partial}{\partial\mu_1}{Q_1|}_{\bomega_1},\frac{\partial}{\partial\nu_1} 
   {Q_1|}_{\bomega_1}\right| = 0. 
   \label{tan_cond}
 \end{equation}
Evaluation of the latter leads to the relations
\bela{munu1}
   (\mu_1-\mu)\left(\nu_1+\frac{\al^2}{\nu}+\al f\right) = 0,\quad
   \left(\mu-\frac{g}{b}\right)\left(\mu_1-\frac{g}{b}\right)- \frac{T}{b^2} = 0,
\ela
where we have assumed that $\bomega$ does not lie on the edge common to $Q$ and $Q_1$, that is, $\nu\neq 0$. For reasons of symmetry, the tangency condition in the other direction produces the pair
\bela{munu2}
   (\nu_2-\nu)\left(\mu_2+\frac{\alb^2}{\mu}+\alb \f\right) = 0,\quad
   \left(\nu-\frac{\g}{\bb}\right)\left(\nu_2-\frac{\g}{\bb}\right)- \frac{\T}{\bb^2} = 0.
\ela
Discrete asymptotic nets which admit envelopes of associated lattices of Lie quadrics may therefore be classified as follows.

\subsection{Discrete Q surfaces}

The solutions $\mu(n_1,n_2)$ and $\nu(n_1,n_2)$ of the systems \eqref{munu1} and \eqref{munu2} determine both the classes of discrete asymptotic nets which admit envelopes and the nature of these envelopes. If 
\bela{E32}
  \mu_1 = \mu\quad\Rightarrow\quad\mu = \mu(n_2)
\ela
then \eqref{munu1}$_1$ is identically satisfied and \eqref{munu1}$_2$ reduces to \eqref{E19}. Hence, any edge $[\bomega,\bomega_1]$ is part of a generator common to the quadrics $Q$ and $Q_1$. Thus, any coordinate polygon $\bomega(n_1,n_2=\mbox{const})$ constitutes a straight line which is a generator shared by the quadrics associated with the corresponding strip of quadrilaterals of the discrete asymptotic net. The discrete asymptotic nets which admit this property are evidently characterised by the constraint
\bela{E33}
  b\mathsf{m}^2 - 2g\mathsf{m} - a = 0,\quad \mathsf{m} = \mathsf{m}(n_2)
\ela
on the Gauss-Mainardi-Codazzi equations \eqref{E15}, \eqref{Teqn} and have been termed {\em discrete semi-Q surfaces} in \cite{McCarthySchief2018}. If, in addition,
\bela{E34}
  \nu_2 = \nu\quad\Rightarrow\quad \nu = \nu(n_1)
\ela
then \eqref{munu2}$_2$ reduces to \eqref{E20} so that the coordinate polygons $\bomega(n_1=\mbox{const},n_2)$ constitute straight lines which are likewise embedded in ``strips'' of lattice Lie quadrics. The corresponding additional constraint on the Gauss-Mainardi-Codazzi equations is given by
\bela{E35}
  \bb\mathsf{n}^2 - 2\g\mathsf{n} - \ab = 0,\quad \mathsf{n} = \mathsf{n}(n_1).
\ela
Discrete asymptotic nets constrained by \eqref{E33} and \eqref{E35} are termed discrete Q surfaces. By construction, their geometric definition is as follows \cite{McCarthySchief2018}.

\begin{definition}
A {\em discrete Q surface} is a discrete asymptotic net which admits an envelope of a lattice of Lie quadrics composed of straight lines which are common generators of the corresponding strips of quadrics.
\end{definition}

It is observed that the envelope $\Omega$ associated with a discrete Q surface is composed of two transversal discrete families of straight lines and may therefore be interpreted as a discretisation of a quadric. Moreover, there exists a unique continuous quadric $Q^\Omega$ which passes through those lines. Thus, the straight coordinate polygons of the envelope $\Omega$ are generators of both strips of lattice Lie quadrics and the quadric $Q^\Omega$. This is the discrete analogue of a property of classical $Q$ surfaces \cite{Bol1954}. 

\subsection{Discrete projective minimal surfaces}

We now assume that either $\mu_1\neq\mu$ or $\nu_2\neq\nu$ so that we may focus on the case $\mu_1\neq\mu$ without loss of generality. Accordingly, \eqref{munu1}$_1$ may be solved for $\nu_1$. Then, either \mbox{$\T=0$} so that $\Sigma$ is of discrete Godeaux-Rozet type and therefore discrete projective minimal or $\T\neq0$, in which case \eqref{munu2}$_2$ may be solved for $\nu_2$ to obtain
\bela{E37}
  \nu_1 = -\frac{\al^2}{\nu}-\al f,\quad\nu_2 = \frac{\g\nu+\ab}{\bb\nu-\g}.
\ela
It turns out that the compatibility condition $\nu_{12} = \nu_{21}$ yields $\Delta_1\T=0$ and, hence, $\Sigma$ is, once again, discrete projective minimal. In fact, it is easy to verify that linearisation of the above ``discrete Riccati equations'' leads to the linear system \eqref{E30} with the identification $\nu = -\rho^{12}_0/\rho^2_0$. In conjunction with the analysis presented in the previous section, we are therefore led to the following theorem (cf.\ \cite{McCarthySchief2018}).

\begin{theorem}
The class of discrete asymptotic nets which admit envelopes of associated lattice Lie quadrics coincides with the class of {\em discrete PMQ surfaces}, that is, discrete projective minimal (PM) and Q surfaces.
\end{theorem}

It is emphasised that, strictly speaking, the above theorem is a statement about the ``local'' nature of a discrete asymptotic net since it may be possible to construct ``hybrids'' of discrete projective minimal and Q surfaces which admit envelopes. In view of the continuum limit, we exclude discrete surfaces of this type. Moreover, for the above theorem to be validated, it is still necessary to show that every discrete projective minimal surface admits an envelope of an associated lattice of Lie quadrics. This will be done as part of the classification given below.

\subsubsection{Generic discrete projective minimal surfaces}

A discrete projective minimal surface is termed {\em generic} if $T\neq 0$ and $\T\neq0$. In this case, for generic initial values $\mu(0,0)$ and $\nu(0,0)$, the system \eqref{E37} and its counterpart
\bela{E38}
 \mu_1 = \frac{g\mu+a}{b\mu-g},\quad\mu_2 = -\frac{\alb^2}{\mu}-\alb \f
\ela
uniquely determine the functions $\mu(n_1,n_2)$ and $\nu(n_1,n_2)$. It is noted that the latter system is compatible modulo the minimality condition $\Delta_2T = 0$. Once again, its linearised version is given by \eqref{E28} with the identification \mbox{$\mu = -\rho^{12}_\infty/\rho^1_\infty$}. Since, by construction, the tangency conditions \eqref{munu1} and \eqref{munu2} are satisfied, the following statement may be made.

\begin{theorem}
Generic projective minimal surfaces admit a two-parameter family of envelopes of the associated lattice of Lie quadrics. An envelope is uniquely determined by (generically) prescribing a vertex on one lattice Lie quadric.
\end{theorem}

It is observed that the systems \eqref{E37} and \eqref{E38} interpreted as defining a relation between points on neighbouring lattice Lie quadrics map generators of a quadric to generators of neighbouring quadrics since $\mu_i$ and $\nu_i$ do not depend on $\nu$ and $\mu$ respectively.

\subsubsection{Generic discrete Godeaux-Rozet surfaces}

A {\em generic} discrete Godeaux-Rozet surface is a discrete projective minimal surface for $T=0$ but $\T\neq0$ or vice versa. Thus, if $\T\neq0$ then the system \eqref{E37} uniquely determines a function $\nu(n_1,n_2)$ for any given generic initial datum $\nu(0,0)$. If, in addition, $T=0$ then \eqref{munu1}$_2$ implies that either $\mu_1=g/b$ or $\mu = g/b$. In both cases, the remaining condition \eqref{munu2}$_1$ is satisfied modulo the Gauss-Mainardi-Codazzi equations subject to $T=0$. Hence, there exist two types of envelopes $\Omega$ and $\tilde{\Omega}$ parametrised by
\bela{E39}
 \begin{aligned}
  \omega &= \r^{12} + \mu\r^1 + \nu\r^2 + \mu\nu\r,\quad &\mu &= \frac{g_{\bar{1}}}{b_{\bar{1}}}\\
  \tilde{\omega} &= \r^{12} + \tilde{\mu}\r^1 + \nu\r^2 + \tilde{\mu}\nu\r,\quad &\tilde{\mu} &= \frac{g}{b}
 \end{aligned}
\ela
with $\nu$ being determined by \eqref{E37}. Moreover, a short calculation reveals that
\bela{E40}
  \omega_1 \sim \r^{12} + \mu_1\r^1 + \nu\r^2 + \mu_1\nu\r = \tilde{\omega}
\ela
since $\tilde{\mu}=\mu_1$ so that $\omega(n_1+1,n_2)\sim\tilde{\omega}(n_1,n_2)$. Hence, for any given solution $\nu$ of the system \eqref{E37}, $\omega$ and $\tilde{\omega}$  represent two parametrisations of the same envelope $\Omega=\tilde{\Omega}$ regarded as sets of points in $\P^3$. Geometrically, this is evident since
\bela{E41}
  b\tilde{\mu}^2 - 2g\tilde{\mu} - a \sim ab+g^2 = T = 0
\ela
so that $\tilde{\mu}$ labels the generator common to the quadrics $Q$ and $Q_1$ along which $Q$ and $Q_1$ touch. Thus, the generator of $Q_1$ labelled by $\mu_1$ coincides with the generator of $Q$ labelled by $\tilde{\mu}$ which, in turn, implies that $\omega_1$ is a point not only on the quadric $Q_1$ but also on the quadric $Q$. 

\begin{theorem}
  A generic discrete Godeaux-Rozet surface admits two one-parameter families of envelopes of the associated lattice of Lie quadrics which coincide.
\end{theorem} 

\subsubsection{Discrete Demoulin surfaces}

By definition, a discrete Demoulin surface is a discrete projective minimal surface for $T=\T=0$. Hence, the tangency conditions \eqref{munu1} and \eqref{munu2} are satisfied if we choose the functions $\mu(n_1,n_2)$ and $\nu(n_1,n_2)$ such that either $\mu_1=g/b$ or $\mu=g/b$ and either $\nu_2=\g/\bb$ or $\nu = \g/\bb$. Accordingly, there exist four envelopes but the same reasoning as in the previous case leads to the following result.

\begin{theorem}
  A discrete Demoulin surface admits four discrete envelopes of the associated lattice of Lie quadrics which coincide.
\end{theorem}


\section{Discrete Demoulin surfaces}

We now touch upon the integrability aspects of discrete projective minimal surfaces. It is known (see, e.g., \cite{RogersSchief2002} and references therein) that, in the classical continuous setting, the {\em Pl\"ucker correspondence} between lines in $\P^3$ and points in the four-dimensional Pl\"ucker quadric embedded in a five-dimensional projective space $\P^5$ plays a key role in this connection. It turns out that the same is true in the discrete setting and the Pl\"ucker correspondence also provides a framework in which discrete analogues of the Wilczynski frame may be identified. This is dealt with in a separate publication \cite{SchiefSzereszewski2018}. Here, we exploit some of these connections which may be verified directly without presenting their derivation. 

\subsection{A discrete Demoulin system. A Wilczynski-type frame}

We begin by introducing a frame $\tilde{\F}=(\tilde{\r},\tilde{\r}^1,\tilde{\r}^2,\tilde{\r}^{12})^T$ in $\P^3$ which obeys the linear system
 \begin{equation}
   \Ft_1 = \tilde{L}\Ft, \qquad \Ft_2 = \tilde{M}\Ft, \label{tildeDem}
\end{equation}  
 where 
\bela{LMtilde}
 \begin{aligned}
    \tilde{L} &= \chi\begin{pmatrix}
        1  &            1                   &    0                &    0  \as
        A  &   A+\dfrac{H_1-1}{H_1(H-1)}     &    \dfrac{AK}{K-1}   &    \dfrac{AK}{K-1} \AS  
      K-1  &           K-1                  &    K                &    K\as
        0  &   \dfrac{(H_1-1)(K-1)}{H_1(H-1)}  &    0              &    \dfrac{(H_1-1)K}{H_1(H-1)}
  \end{pmatrix}\as
  \tilde{M} &= \bar{\chi}\begin{pmatrix}
    1  &          0        &    1                &    0  \as
  H-1  &         H         &    H-1                &    H\as 
    Q  &  \dfrac{QH}{H-1}  &   Q+\dfrac{K_2-1}{K_2(K-1)}  &    \dfrac{QH}{H-1}  \AS  
    0  &          0        &  \dfrac{(K_2-1)(H-1)}{K_2(K-1)} &  \dfrac{(K_2-1)H}{K_2(K-1)}
\end{pmatrix}.
\end{aligned}
\ela
The compatibility condition $\tilde{L}_2\tilde{M} = \tilde{M}_1\tilde{L}$ gives rise to the system of equations
\bela{dem12}
\begin{aligned}
  H_{12}&=-\frac{K(H-1)(K_2-1)}{K\big[H(H_1-1)(H_2-1)-(H-1)\big](K_2-1)-AQH_1 K_2 (H_2-1)}\\
 K_{12}&=-\frac{H(K-1)(H_1-1)}{H\big[K(K_1-1)(K_2-1)-(K-1)\big](H_1-1)-AQH_1 K_2 (K_1-1)}
\end{aligned}
\ela
and 
\begin{equation}
  A_2=\frac{H_1}{K}A, \quad Q_1=\frac{K_2}{H}Q \label{Demcomconds1}
\end{equation}   
together with
\bela{E42}
  \chi_2 \bar{\chi} H = \bar{\chi}_1 \chi K.
\ela
Even though the compatibility of the system \eqref{tildeDem} does not require specific knowledge of the individual functions $\chi$ and $\bar{\chi}$, geometric considerations dictate that 
\begin{equation}\label{qqb}
   \chi^2=\frac{(1-H)H_1}{(1-H_1)K} , \quad  \bar{\chi}^2=\frac{(1-K)K_2}{(1-K_2)H}
\end{equation}
which may be shown to obey  $(\chi_2 \bar{\chi} H)^2 = (\bar{\chi}_1 \chi K)^2$ modulo \eqref{dem12} so that the auxiliary condition \eqref{E42} is satisfied if the signs of $\chi$ and $\bar{\chi}$ are chosen appropriately. 

It turns out that the coupled system \eqref{dem12}, \eqref{Demcomconds1} constitutes a discrete version of the Demoulin system \eqref{gmc}$_{\alpha=\beta=0}$ in the form \cite{Finikov1937}
\bela{E43}
  {(\ln h)}_{xy} = h - \frac{\mathsf{a}\mathsf{q}}{hk},\quad {(\ln k)}_{xy} = k - \frac{\mathsf{a}\mathsf{q}}{hk},\quad \mathsf{a}_y=0,\quad \mathsf{q}_x = 0
\ela
with the identification
\bela{E43a}
  p =  \frac{\mathsf{a}}{k},\quad q = \frac{\mathsf{q}}{h},\quad \mathcal{A} = -\mathsf{a},\quad \mathcal{B} = -\mathsf{q}.
\ela
Indeed, the expansion
\bela{E44}
  H = 1 + \frac{\epsilon\delta}{2}h,\quad K = 1 + \frac{\epsilon\delta}{2}k,\quad A = \frac{\epsilon^3}{2}\mathsf{a},\quad Q = \frac{\delta^3}{2}\mathsf{q},\quad x = \epsilon n_1,\quad y = \delta n_2
\ela
leads to the Demoulin system in the limit $\epsilon,\delta\rightarrow 0$. In order to reveal the nature of the frame $\tilde{\F}$, we note that \eqref{qqb} implies the expansion
\bela{E45}
  \chi = 1 - \frac{\epsilon}{2}\frac{h_x}{h} + \cdots,\quad \bar{\chi} = 1 - \frac{\delta}{2}\frac{k_y}{k} + \cdots
\ela
so that the scaling
\bela{E46}
 \tilde{\r}\rightarrow\tilde{\r},\quad\tilde{\r}^1\rightarrow\epsilon\tilde{\r}^1,\quad\tilde{\r}^2\rightarrow\delta\tilde{\r}^2,\quad\tilde{\r}^{12}\rightarrow\epsilon\delta\tilde{\r}^{12}
\ela
leads to the continuum limit
\bela{E47}
 \tilde{\F}_x = \frac{1}{2}\begin{pmatrix}
  -\dfrac{h_x}{h} & 2 & 0 & 0\as
  0 &\dfrac{h_x}{h} & 2\dfrac{\mathsf{a}}{k} & 0\As
  k & 0 & -\dfrac{h_x}{h} & 2\as
  0 & k & 0 &\dfrac{h_x}{h}\end{pmatrix}\tilde{\F},\quad 
  \tilde{\F}_y = \frac{1}{2}\begin{pmatrix}
  -\dfrac{k_y}{k} & 0 & 2 & 0\as
  h & -\dfrac{k_y}{k} & 0 & 2\As
  0 & 2\dfrac{\mathsf{q}}{h} &\dfrac{k_y}{k} & 0\as
  0 & 0 & h &\dfrac{k_y}{k}\end{pmatrix}\tilde{\F}.
\ela
In particular, it is seen that
\bela{E48}
 \begin{gathered}
  \tilde{\r}^1 = \tilde{\r}_x  + \frac{1}{2}\frac{h_x}{h}\tilde{\r},\quad  \tilde{\r}^2 = \tilde{\r}_y  + \frac{1}{2}\frac{k_y}{k}\tilde{\r}\as
\tilde{\r}^{12} = \tilde{\r}_{xy} + \frac{1}{2}\frac{k_y}{k}\tilde{\r}_x + \frac{1}{2}\frac{h_x}{h}\tilde{\r}_y + \left(\frac{1}{4}\frac{h_xk_y}{hk} - \frac{1}{2}\frac{\mathsf{a}\mathsf{q}}{hk}\right)\tilde{\r} 
 \end{gathered} 
\ela
so that comparison with \eqref{C2} shows that $\tilde{\F}$ constitutes a Wilczynski frame for Demoulin surfaces.

\subsection{Connection with the canonical frame}

In order to demonstrate that the discrete frame $\tilde{\F}$ encodes discrete Demoulin surfaces, it is required to find a gauge transformation which maps the frame $\tilde{\F}$ to the canonical frame $\F$ for discrete Demoulin surfaces. If we introduce a gauge matrix $G$ according to
\begin{equation}\label{gauge}   
  \Ft =  G\F = \begin{pmatrix}
       \xi  &            0            &     0      &    0    \\
      -\xi  &           \ka           &     0      &    0    \\  
      -\xi  &            0            &     \kab   &    0    \\
       \xi  &          -\ka           &    -\kab   &   1/\xi \\
\end{pmatrix} \F,
\end{equation}  
wherein the functions $\ka,\kab$ and $\xi$ are constrained by
\begin{equation}\label{E48a}
     \xi_1 = \frac{\ka\xi}{\ka_1 K}, \quad \xi_2=\frac{\kab\xi}{\kab_2 H}, \quad \ka\kab=1,
\end{equation}
then the transformed matrices $L$ and $M$ are given by
\bela{E49}
  L = G_1^{-1}\tilde{L}G,\quad M = G_2^{-1}\tilde{M}G.
\ela
It is noted that, on elimination of (for instance) $\kab$, the remaining pair of equations \eqref{E48a}$_{1,2}$ may be regarded as a linear system for $\xi$, the compatibility of which leads to a nonlinear equation for~$\ka$.

Now, it may be directly verified that evaluation of \eqref{E49} produces matrices $L$ and $M$ which are precisely of the form \eqref{Feq1} with
\bela{E51}
 \begin{aligned}
   a &= \frac{\chi\xi\ka^2}{\ka_1}\frac{A}{(1-K)K},\val   &  \ab &= \frac{\bar{\chi}\xi\kab^2}{\kab_2}\frac{Q}{(1-H)H},\val &
   b &= \frac{\chi}{\xi\ka_1}\frac{AK}{K-1}, \val     &  \bb &= \frac{\bar{\chi}}{\xi\kab_2}\frac{QH}{H-1}\\
  f &= \frac{\chi\ka}{\ka_1}\frac{1-HH_1}{(1-H)H_1}, &  \f &= \frac{\bar{\chi}\kab}{\kab_2}\frac{1-KK_2}{(1-K)K_2}, &
  g &= \frac{\chi\ka}{\ka_1}\frac{A}{1-K}, & \g &= \frac{\bar{\chi}\kab}{\kab_2}\frac{Q}{1-H}\\
  \al &= \frac{\xi}{\chi\ka_1 K},  &  \alb &= \frac{\xi}{\bar{\chi}\kab_2 H}.
 \end{aligned}
\ela
Accordingly, the frame $\F$ is indeed canonical. Moreover, it is evident that
\bela{E52}
  ab + g^2 = 0,\quad \ab\bb + \g^2 = 0 
\ela
so that $\F$ is the canonical frame associated with discrete Demoulin surfaces. Finally, the invariance
\bela{E53}
  A \rightarrow \lambda A,\quad Q\rightarrow \lambda^{-1}Q
\ela
of the discrete Demoulin system \eqref{dem12}, \eqref{Demcomconds1} injects a parameter into the frame equations \eqref{tildeDem} and corresponds to the scaling \eqref{E23} as may be inferred from \eqref{E51}.

\subsection{Discrete Tzitz\'eica surfaces}

In order to make contact with some established results in discrete differential geometry, we observe that the discrete Demoulin system admits the reduction $K=H$, leading to the standard discretisation \cite{Schief1999}
\bela{tzitzeica}
 \begin{gathered}
  H_{12}=-\frac{H(H-1)}{H[H(H_1-1)(H_2-1)-(H-1)]-AQH_1 H_2}\\
   A_2=\frac{H_1}{H}A, \quad Q_1=\frac{H_2}{H}Q
 \end{gathered}
\end{equation}  
of the classical {\em Tzitz\'eica equation} given by \eqref{E43} with $h=k$ governing affine spheres (see \cite{Schief2000} and references therein). Moreover, \eqref{qqb} shows that $\chi$ and $\bar{\chi}$ are of the algebraic form \mbox{$\chi = \varphi/\varphi_1$} and $\bar{\chi} = \varphi/\varphi_2$ so that introduction of the scaled Wilczynski frame $\hat{\F} = \varphi\tilde{\F}$ leads to the frame equations
\bela{LMhat}
  \begin{aligned}
    \hat{\F}_1 &= \begin{pmatrix}
        1  &            1                   &    0                &    0  \as
        A  &   A+\dfrac{H_1-1}{H_1(H-1)}     &    \dfrac{AH}{H-1}   &    \dfrac{AH}{H-1} \AS  
       H-1  &           H-1                  &    H                &    H\as
        0  &   \dfrac{H_1-1}{H_1}  &    0              &    \dfrac{(H_1-1)H}{H_1(H-1)}
\end{pmatrix}\hat{\F}\as
    \hat{\F}_2 &= \begin{pmatrix}
    1  &          0        &    1                &    0  \as
  H-1  &         H         &    H-1                &    H\as 
    Q  &  \dfrac{QH}{H-1}  &   Q+\dfrac{H_2-1}{H_2(H-1)}  &    \dfrac{QH}{H-1}  \AS  
    0  &          0        &  \dfrac{H_2-1}{H_2} &  \dfrac{(H_2-1)H}{H_2(H-1)}
\end{pmatrix}\hat{\F}.
  \end{aligned}
\ela
These may be formulated as
\bela{E54}
  \hat{\r}^1 = \hat{\r}_1 - \hat{\r},\quad \hat{\r}^2 = \hat{\r}_2 - \hat{\r},\quad \hat{\r}^{12} = \frac{1}{H}\hat{\r}_{12} - \hat{\r}_1 - \hat{\r}_2 + \hat{\r}
\ela
together with
\bela{E55}
 \begin{aligned}
  \hat{\r}_{11} - \hat{\r}_1 & = \frac{H_1-1}{H_1(H-1)}(\hat{\r}_1 - \hat{\r}) + \frac{A}{H-1}(\hat{\r}_{12} - \hat{\r}_1)\\
  \hat{\r}_{22} - \hat{\r}_2 & = \frac{H_2-1}{H_2(H-1)}(\hat{\r}_2 - \hat{\r}) + \frac{Q}{H-1}(\hat{\r}_{12} - \hat{\r}_2).
 \end{aligned}
\ela
If we now introduce affine coordinates $\hat{\r} = (\r^a,1)$ and note that \eqref{E55} implies that
\bela{E56}
  \Delta_i\left(\frac{1}{H-1}[\hat{\r}_{12} + \hat{\r} - H(\hat{\r}_1 + \hat{\r}_2)]\right) = 0
\ela
then integration of the latter and application of an appropriate translation $\hat{\r} \rightarrow\hat{\r} + \mbox{\bf const}$ result in the standard Gauss-Weingarten equations for {\em discrete affine spheres} in (centro-)affine geometry, namely \cite{BobenkoSchief1999}
\bela{E57}
 \begin{aligned}
  \r^a_{11} - \r^a_1 & = \frac{H_1-1}{H_1(H-1)}(\r^a_1 - \r^a) + \frac{A}{H-1}(\r^a_{12} - \r^a_1)\\
  \r^a_{12} + \r^a & = H(\r^a_1 + \r^a_2)\\
  \r^a_{22} - \r^a_2 & = \frac{H_2-1}{H_2(H-1)}(\r^a_2 - \r^a) + \frac{Q}{H-1}(\r^a_{12} - \r^a_2).
 \end{aligned}
\ela
We therefore conclude that the frame equations \eqref{LMhat} or, equivalently, \eqref{E55} encode the class of projective transforms of discrete affine spheres. It is emphasised that the Gauss-Weingarten equations \eqref{E57} for discrete affine spheres have originally been derived in \cite{BobenkoSchief1999} in an entirely different geometric manner.

In \cite{McCarthySchief2018}, the specialisation of discrete Demoulin surfaces to (projective transforms of) discrete Tzitz\'eica surfaces has been characterised geometrically in terms of discrete line congruences. This result will now be complemented by deriving the algebraic reduction of the discrete Gauss-Mainardi-Codazzi equations \eqref{E15} with
\bela{E58}
 a = -\frac{g^2}{b},\quad \ab = -\frac{\g^2}{\bb}
\ela
which corresponds to discrete Tzitz\'eica surfaces. In this connection, the key observation is that the system \eqref{E51} implies the identity
\bela{E59}
  \left(\frac{g}{b}\right)_{12} \left(\frac{\g}{\bb}\right)_1 = \left(\frac{\g}{\bb}\right)_{12} \left(\frac{g}{b}\right)_2
\ela
which is now regarded as a constraint on the Gauss-Mainardi-Codazzi equations. It implies that there exists a potential $\tau$ such that
\bela{E60}
  \frac{g}{b} = \frac{\tau_{12}}{\tau_1},\quad\frac{\g}{\bb} = \frac{\tau_{12}}{\tau_2}.
\ela
If we regard the latter as a parametrisation of the functions $b$ and $\bb$ then the Gauss-Mainardi-Codazzi equations \eqref{E15}$_{4,8}$ become
\bela{E61}
  \tau_{112} + \alpha (f\tau_{12} + \alpha\tau_2) = 0,\quad \tau_{122} + \alb(\f\tau_{12} + \alb\tau_1) = 0
\ela
which may be simplified to
\bela{E62}
  \tau_{11} + \alpha (f\tau_1 + \alpha\tau) = 0,\quad \tau_{22} + \alb(\f\tau_2 + \alb\tau) = 0
\ela
since \eqref{E61} is seen to coincide with \eqref{E62}$_1$ and \eqref{E62}$_2$ shifted in the directions $n_2$ and $n_1$ respectively. Finally, if we regard \eqref{E62} as a linear system for the function $\tau$ then it is required to examine the corresponding compatibility condition $\tau_{1122}=\tau_{2211}$ or, equivalently, the compatibility condition ${(\tau_{112})}_2 = {(\tau_{122})}_1$ associated with \eqref{E61}. It turns out that this compatibility condition does not generate any additional relation and, hence, the constraint \eqref{E59} has been shown to be admissible.

The nature of the function $\tau$ is revealed by verifying that
\bela{E63}
   \mathsf{s}(n_1) = \alpha g \left(\frac{\tau_2}{\tau_{12}} - \frac{\tau}{\tau_1}\right),\quad  \bar{\mathsf{s}}(n_2) = \alb \g \left(\frac{\tau_1}{\tau_{12}} - \frac{\tau}{\tau_2}\right)
\ela
constitute first integrals so that $\tau$ may be shown to obey the equation
\bela{E64}
  \left|\bear{ccc}\tau&\tau_1&\tau_{11}\\ \tau_2&\tau_{12}&\tau_{112}\\ \tau_{22}&\tau_{122}&\tau_{1122}\ear\right| + \mathsf{s}\bar{\mathsf{s}}\tau_{12}^3 = 0.
\ela
The latter is known to be the {\em$\tau$-function} representation \cite{Schief1999} of the discrete Tzitz\'eica system \eqref{tzitzeica}. Indeed, the pair \eqref{tzitzeica}$_{2,3}$ may be regarded as the compatibility conditions for the existence of a function $\tau$ defined according to
\bela{E65}
  \tau_{11} = \mathsf{s}\frac{\tau_1^2}{\tau A},\quad \tau_{12} = \frac{\tau_1\tau_2}{\tau H},\quad \tau_{22} = \bar{\mathsf{s}}\frac{\tau_2^2}{\tau Q}.
\ela
Hence, if we regard the above relations as a parametrisation of the functions $H$ and $A,Q$ in terms of $\tau$ then the remaining discrete Tzitz\'eica equation \eqref{tzitzeica}$_1$ coincides with \eqref{E64}.

\section{A B\"acklund transformation for the discrete Demoulin system}

We now exploit the afore-mentioned Pl\"ucker correspondence to derive a {\em B\"acklund transformation} \cite{RogersSchief2002} for the discrete Demoulin system \eqref{dem12}, \eqref{Demcomconds1}, thereby rendering discrete Demoulin surfaces integrable. Thus, a line in $\P^3$ passing through two points $\bolda = (a^0,a^1,a^2,a^3)$ and \mbox{$\boldb=(b^0,b^1,b^2,b^3)$} may be represented by the exterior product
\bela{E66}
  \bolda\wedge\boldb = (p^{01},p^{23},p^{02},p^{13},p^{03},p^{12}),\qquad p^{ik} = \left|\bear{cc} a^i & a^k\\ b^i & b^k\ear\right|.
\ela
The identity
\bela{E68}
  p^{01}p^{23} - p^{02}p^{13} + p^{03}p^{12} = 0
\ela
shows that the {\em Pl\"ucker coordinates} $p^{ik}$ of the line constitute homogeneous coordinates of a four-dimensional quadric which is embedded in a five-dimensional projective space $\P^5$. As in the continuous setting \cite{RogersSchief2002}, we introduce quantities $\bvpt,\bpsit\in\P^5$ defined by
 \begin{equation}\label{E69}
   \bvpt = \frac{1}{2}\left(\btr^1 \wedge \btr^2 + \btr\wedge\btr^{12}\right), \quad  
   \bpsit = \frac{1}{2}\left(\btr^2 \wedge \btr^1 + \btr\wedge\btr^{12}\right).
 \end{equation}
The Wilczynski frame equations \eqref{tildeDem} then give rise to the linear system
\bela{psi12}
\begin{aligned}
  \bvpt_{11} &= \bvpt_1+\frac{1-H_1}{(1-H)H_1} (\bvpt_1-\bvpt) + \frac{\lambda A}{K-1}(\bpsit_{12}-\bpsit_1)\\
  \bvpt_{22} &= \bvpt_2+\frac{1-H_2}{(1-H)H_2} (\bvpt_2-\bvpt) + \frac{B}{\lambda(K-1)}(\bpsit_{12}-\bpsit_2)\\
  \bvpt_{12} &= -\bvpt + H(\vpt_1+\vpt_2)\\
  \bpsit_{11} &= \bpsit_1+\frac{1-K_1}{(1-K)K_1} (\bpsit_1-\bpsit) + \frac{\lambda P}{H-1}(\bvpt_{12}-\bvpt_1)\\
  \bpsit_{22} &= \bpsit_2+\frac{1-K_2}{(1-K)K_2} (\bpsit_2-\bpsit) + \frac{Q}{\lambda(H-1)}(\bvpt_{12}-\bvpt_2)\\
  \bpsit_{12} &= -\bpsit + K(\bpsit_1+\bpsit_2),
\end{aligned}
\ela
where
 \begin{equation}
  B = \frac{(H_2-1)(K-1)K_2}{(H-1)(K_2-1)H_2}Q, \quad P = \frac{(H-1)(K_1-1)H_1}{(H_1-1)(K-1)K_1}A \label{BP}
 \end{equation}
and the parameter $\lambda$ has been inserted via the scaling \eqref{E53} so that the geometrically relevant case is obtained by setting $\lambda=1$. By construction, the associated compatibility conditions reproduce the discrete Demoulin system \eqref{dem12}, \eqref{Demcomconds1}. It is noted that the quantities $P$ and $B$ obey the analogue of the pair \eqref{Demcomconds1}, namely,
\begin{equation}
  P_2 = \frac{K_1}{H}P, \quad B_1 = \frac{H_2}{K}B. \label{B1P2}
\end{equation}
Moreover, in the admissible reduction $H=K$, $\bvpt=\bpsit$ and $P=A$, $B=Q$, the linear system \eqref{psi12} coincides formally with the affine Gauss-Weingarten equations \eqref{E57} for discrete Tzitz\'eica surfaces. It should also be mentioned that the Wilczynski frame $\Ft$ may be reconstructed from
the pair $(\bvpt,\bpsit)$.

We are now in a position to formulate the B\"acklund transformation.

\begin{theorem}
The discrete Demoulin system \eqref{dem12}, \eqref{Demcomconds1} and \eqref{BP} together with the associated linear system \eqref{psi12} is invariant under the substitution
\bela{E70}
 \begin{aligned}
  \bvpt' & = \frac{\bS}{\vpt^\circ},\val &  H' &= \frac{\vpt_1^\circ\vpt_2^\circ}{\vpt_{12}^\circ\vpt^\circ}H,\val & A' &= \frac{\vpt_1^\circ\psit_1^\circ}{\vpt_{11}^\circ\psit^\circ}A,\val & B' &= \frac{\vpt_2^\circ\psit_2^\circ}{\vpt_{22}^\circ\psit^\circ} B\\
  \bpsit' & = \frac{\bT}{\psit^\circ},\val &  K' &= \frac{\psit_1^\circ\psit_2^\circ}{\psit_{12}^\circ\psit^\circ}K,\val & P' &= \frac{\psit_1^\circ\vpt_1^\circ}{\psit_{11}^\circ\vpt^\circ}P,\val & Q' &=  \frac{\psit_2^\circ\vpt_2^\circ}{\psit_{22}^\circ\vpt^\circ}Q,
 \end{aligned} 
\ela
where
\bela{E71}
 \begin{aligned}
 \bS = c^3\vpt^\circ\bvpt - 2c^1\psit^\circ\bpsit & - \frac{H}{H-1}(c^2\Delta_1\vpt^\circ\Delta_2\bvpt + c^0\Delta_2\vpt^\circ\Delta_1\bvpt)\\
                                                                       & + \frac{K}{K-1}c^1(\Delta_1\psit^\circ\Delta_2\bpsit + \Delta_2\psit^\circ\Delta_1\bpsit)\\[1mm]
  \bT = c^3\psit^\circ\bpsit - 2c^1\vpt^\circ\bvpt & - \frac{K}{K-1}(c^2\Delta_1\psit^\circ\Delta_2\bpsit + c^0\Delta_2\psit^\circ\Delta_1\bpsit)\\
                                                                       & + \frac{H}{H-1}c^1(\Delta_1\vpt^\circ\Delta_2\bvpt + \Delta_2\vpt^\circ\Delta_1\bvpt)
 \end{aligned}
\ela
with coefficients
\bela{E72}
  c^0 = \frac{\lambda_\circ^2}{\lambda^2-\lambda_\circ^2},\quad
  c^1 = \frac{\lambda_\circ\lambda}{\lambda^2-\lambda_\circ^2},\quad 
  c^2 = \frac{\lambda^2}{\lambda^2-\lambda_\circ^2},\quad
  c^3 = \frac{\lambda^2+\lambda_\circ^2}{\lambda^2-\lambda_\circ^2}
\ela
and $(\vpt^\circ,\psit^\circ)$ is a scalar solution of the linear system \eqref{psi12} corresponding to a parameter $\lambda_\circ$ subject to the admissible constraint
\bela{E73}
 {(\vpt^\circ)}^2 - \frac{H}{H-1}\Delta_1\vpt^\circ\Delta_2\vpt^\circ  =  {(\psit^\circ)}^2  - \frac{K}{K-1}\Delta_1\psit^\circ\Delta_2\psit^\circ.
\ela
\end{theorem}

\begin{proof}
The {\em discrete Moutard equations} \eqref{psi12}$_{3,6}$ are known to be invariant under the {\em discrete Moutard transformation} \cite{NimmoSchief1997,BobenkoSchief1999} \eqref{E70}$_{1,2,5,6}$, where $\vpt^\circ$ and $\psit^\circ$ are scalar solutions of the discrete Moutard equations and the ``bilinear potentials'' $\bS$ and $\bT$ are defined by the compatible system
\bela{E74}
 \begin{aligned}
  \Delta_1\bS &= \vpt^\circ\bvpt_1 - \vpt_1^\circ\bvpt,\val & \Delta_2\bS &= \vpt_2^\circ\bvpt - \vpt^\circ\bvpt_2\\
  \Delta_1\bT &= \psit^\circ\bpsit_1 - \psit_1^\circ\bpsit,\val & \Delta_2\bT &= \psit_2^\circ\bpsit - \psit^\circ\bpsit_2.
 \end{aligned}
\ela
Now, if $(\vpt^\circ,\psit^\circ)$ is also a solution of the remaining linear equations \eqref{psi12}$_{1,2,4,5}$ then the solution of \eqref{E74} may be verified to be given by \eqref{E71} up to two additive constants of integration. These two constants are required to vanish if the complete linear system \eqref{psi12} is to be invariant. The quantities $A,B$ and $P,Q$ transform according to \eqref{E70}$_{3,4,7,8}$ provided that the constraint \eqref{E73} holds. It turns out that the difference of the left- and right-hand sides of this constraint is constant and, hence, it is possible to choose $\vpt^\circ$ and $\psit^\circ$ in such a manner that this constant is zero.
\end{proof}

In conclusion, it is observed that the B\"acklund transforms \eqref{E70}$_{2,3,4,6,7,8}$ suggest introducing the parametrisation
\bela{E75}
 \begin{aligned}
  H &= \frac{\tau_1\tau_2}{\tau_{12}\tau},   &\val   A &= \frac{\tau_1\sigma_1}{\tau_{11}\sigma},   &\val   
      B &= \frac{\tau_2\sigma_2}{\tau_{22}\sigma}\\[1mm]
  K &= \frac{\sigma_1\sigma_2}{\sigma_{12}\sigma},   &   P &= \frac{\sigma_1\tau_1}{\sigma_{11}\tau},   &   
      Q &= \frac{\sigma_2\tau_2}{\sigma_{22}\tau}
 \end{aligned}
\ela
so that the transformation formulae 
\bela{E76}
  \tau' = \vpt^\circ\tau,\quad \sigma' = \psit^\circ\sigma
\ela
result. Indeed, the pairs \eqref{Demcomconds1} and \eqref{B1P2} may be regarded as the compatibility conditions for the existence of the potentials $\tau$ and $\sigma$. The discrete Demoulin system \eqref{dem12}  then adopts the form 
\bela{E77}
 \begin{aligned}
   \begin{vmatrix}
     \s_2     & \s_{12} \\
     \s_{22}  & \s_{122}
   \end{vmatrix}
   \begin{vmatrix}
     \tau      &    \tau_1     &    \tau_{11}\\
     \tau_2    &    \tau_{12}  &    \tau_{112}\\ 
     \tau_{22} &    \tau_{122} &    \tau_{1122}\\
   \end{vmatrix} &+ 
   \begin{vmatrix}
     \tau_2     & \tau_{12} \\
     \tau_{22}  & \tau_{122}
   \end{vmatrix} \s_{12}^2\tau_{12} = 0 \\[1mm]
   \begin{vmatrix}
     \tau_1     & \tau_{11} \\
     \tau_{12}  & \tau_{112}
   \end{vmatrix}
   \begin{vmatrix}
     \s      &    \s_2     &    \s_{22}\\
     \s_1    &    \s_{12}  &    \s_{122}\\ 
     \s_{11} &    \s_{112} &    \s_{1122}\\
   \end{vmatrix} &+ 
   \begin{vmatrix}
     \s_1     & \s_{11} \\
     \s_{12}  & \s_{112}
   \end{vmatrix} \tau_{12}^2\s_{12} = 0
 \end{aligned}
\ela
which constitutes a two-component generalisation of the discrete Tzitz\'eica equation \eqref{E64}. However, it is important to note that the above system is constrained by the relations \eqref{BP} which read
\bela{E78}
 \begin{aligned}
  \begin{vmatrix}
     \tau       & \tau_{1} \\
     \tau_{2}   & \tau_{12}
  \end{vmatrix}
  \begin{vmatrix}
     \s_2     & \s_{22} \\
     \s_{12}  & \s_{122}
   \end{vmatrix} &=
  \begin{vmatrix}
     \s         & \s_{1} \\
     \s_{2}     & \s_{12}
  \end{vmatrix}
  \begin{vmatrix}
     \tau_2     & \tau_{22} \\
     \tau_{12}  & \tau_{122}
   \end{vmatrix}                   \\[1mm]
   \begin{vmatrix}
     \tau       & \tau_{1} \\
     \tau_{2}   & \tau_{12}
  \end{vmatrix}
  \begin{vmatrix}
     \s_1     & \s_{11} \\
     \s_{12}  & \s_{112}
   \end{vmatrix} &=
  \begin{vmatrix}
     \s         & \s_{1} \\
     \s_{2}     & \s_{12}
  \end{vmatrix}
  \begin{vmatrix}
     \tau_1     & \tau_{11} \\
     \tau_{12}  & \tau_{112}
   \end{vmatrix} .
   \end{aligned}
\ela
By construction, the latter are compatible with the pair \eqref{E77} which may be seen directly by verifying that \eqref{E78}$_1$ and \eqref{E78}$_2$ are invariant under shifts in the directions $n_1$ and $n_2$ respectively modulo \eqref{E77}. In the case of the Tzitz\'eica reduction $\tau=\sigma$, these constraints are identically satisfied.  

\bigskip

\noindent
{\bf Funding.} This work was supported by the Australian Research Council (ARC Discovery Project DP140100851).

\end{document}